\documentclass[12pt]{amsart}
\usepackage[colorlinks=true,citecolor=black,linkcolor=black,urlcolor=blue]{hyperref}
\usepackage{amsmath}
\usepackage[english,  activeacute]{babel}
\usepackage[utf8]{inputenc}
\usepackage{amssymb}
\usepackage{amsthm}
\usepackage{graphics,graphicx}
\usepackage{enumerate}
\usepackage{array}
\usepackage{bm}
\usepackage{a4wide}
\usepackage{color, url}
\usepackage{float}
\usepackage{pgf, tikz}

\usepackage{algorithm}
\usepackage[noend]{algpseudocode}

\theoremstyle{plain}
\newtheorem{theorem}{Theorem}[section]

\newtheorem{corollary}[theorem]{Corollary}

\theoremstyle{definition}

\newcommand{\Z}{{\mathbb Z}}
\newcommand{\Sg}{{\mathfrak S}}
\newcommand{\Pa}{{\mathcal P}}
\newcommand{\T}{{\mathcal T}}
\def\sttf2#1#2{\left[\!\!\left[#1\atop#2\right]\!\!\right]}

\title[Modular symmetric functions]{New modular symmetric function and its applications: Modular $s$-Stirling numbers}

\author{ Bazeniar Abdelghafour}
\address{\noindent    University Center of Abdelhafid Boussouf, Mila. Department of Mathematics, University of Mohamed Seddik Benyahia,  LMAM laboratory, Jijel 18000, Algeria}
\email{baz.abdelghafour@gmail.com}

\author{Moussa Ahmia}
\address{\noindent  Department of Mathematics, University of Mohamed Seddik Benyahia,  LMAM laboratory, Jijel 18000, Algeria}
\email{ahmiamoussa@gmail.com}

\author{Jos\'e L. Ram\'{\i}rez}
\address{\noindent Departamento de Matem\'aticas, Universidad Nacional de Colombia, Bogot\'a,  Colombia}
\email{jlramirezr@unal.edu.co}
\urladdr{http://sites.google.com/site/ramirezrjl}

\author{ Diego Villamizar}
\address{\noindent   Department of Mathematics and Systems Analysis, Aalto University, 00076 Aalto, Finland}
\email{diego.villamizarrubiano@aalto.fi}
\urladdr{https://sites.google.com/view/dvillami/}

\date{\today}

\begin{document}
\begin{abstract}

In this paper, we consider a  generalization of the Stirling number
sequence of both  kinds by using a specialization of a new family of symmetric functions. We give  combinatorial interpretations for this symmetric functions by means of weighted lattice path and tilings.  We also present some new convolutions  involving the complete and elementary symmetric functions. Additionally, we introduce different families of set partitions to give combinatorial   interpretations for the modular $s$-Stirling numbers.
\end{abstract}

\maketitle

\section{Introduction}
\label{S1}
A symmetric function is homogeneous of degree $k$ if every monomial in it has total degree $k$. Symmetric functions are ubiquitous in mathematics and mathematical physics. For example, they appear in elementary algebra (e.g. Vi\`ete's theorem), representation theories of symmetric groups, and general linear groups over the complex numbers or finite fields. They are also important objects to study in algebraic combinatorics.

A \emph{set partition} of a set $[n] :=\{1, 2, \ldots, n\}$ is a collection of non-empty disjoint subsets, called \emph{blocks}, whose union is $[n]$. Let ${n \brace k}$ denote the number of set partitions of $[n]$ into $k$ non-empty blocks. This sequence is called the \emph{Stirling numbers of the second kind}. Similarly, let ${n \brack k}$ denote the number of permutations of $[n]$ into $k$ cycles. This sequence is called the \emph{Stirling numbers of the first kind}.  The literature contains several generalizations of Stirling numbers of both kinds; see for example 
\cite{Mansour, Mezo, Caicedo}. 

Given a set of variables $x_1,x_2,\ldots,x_n$, the $k$-th \emph{elementary and complete symmetric polynomials} are defined, respectively, by
\begin{align*}
e_k(x_1,x_2,\dots,x_n)&=\sum_{1\leq i_1<i_2<\cdots<i_k\leq n } x_{i_1}x_{i_2}\cdots x_{i_k}, \qquad 1 \leq k \leq n,\\
h_k(x_1,x_2,\dots,x_n)&=\sum_{1\leq i_1\leq i_2\leq \cdots \leq i_k\leq n } x_{i_1}x_{i_2}\cdots x_{i_k}, \qquad k \geq 1,
\end{align*}
with initial conditions $e_0(x_1,x_2,\dots,x_n)=h_0(x_1, x_2,\dots, x_n)=1$. Note that \linebreak $e_k(x_1,x_2,\dots,x_n)=0$ if $k>n$.  The generating functions for the $e_k$ and $h_k$ are given by the expressions 
\begin{align*}
\sum_{k=0}^ne_k(x_1,x_2,\dots,x_n)z^k&=\prod_{i=1}^n(1+x_iz),\\
\sum_{k\geq 0}h_k(x_1,x_2,\dots,x_n)z^k&=\prod_{i=1}^n\frac{1}{1-x_iz}.
\end{align*}
A variety of  combinatorial sequences can be obtained as evaluations of the symmetric polynomials at specific points (cf. \cite{Egge, Mack}). Particularly, the Stirling numbers of both kinds are given by 
$$h_k(1, 2, \dots, n)={n+k \brace k} \quad \text{and}\quad  e_k(1, 2, \dots, n)={n+1 \brack n+1-k}.$$ 

In this work, we introduce an extension of the Stirling numbers of both kinds, called \emph{$s$-modular Stirling numbers}, by introducing a new class of symmetric functions, and considering these new sequences as specializations of this symmetric function. We give a combinatorial interpretation of these symmetric functions by using weighted lattice path and tilings. Similar symmetric functions were studied by Doty and Walker under the name of modular complete symmetric polynomials \cite{DW}.  Most recently, Ahmnia and Merca \cite{MMerc} introduced a variation of these symmetric functions.  Independently, Grinberg \cite{Grinberg} and Fu and Mei   \cite{FM} introduced the same concept under the name of Petrie symmetric functions and   truncated symmetric functions, respectively. Finally, we use set partitions to give a combinatorial interpretation to the $s$-modular Stirling numbers.  Among other things, we give an interpretation (probably new) of the Stirling numbers of first kind in terms of set partitions. We also give a relationship with the Stirling numbers with higher level. This last sequence was recently studied in the context of special polynomials \cite{KomPit}.

\section{Definitions and properties}
 Let $s\geq 1$ be a positive integer. We define a {\it modular symmetric function} by
\begin{equation}\label{emf}
M_{k}^{(s)}(n):=M_{k}^{(s)}(x_{1},\ldots,x_{n})=\sum_{_{\substack{ a_{1}+\cdots +a_{n}=k  \\ a_{1},\ldots,a_{n}\equiv (0,1)\mod(s+1)}}}x_{1}^{a_{1}}\cdots x_{n}^{a_{n}},
\end{equation}
with $M_{k}^{(1)}(n)=h_{k}(n)$ and $M_{k}^{(s)}(0)=\delta_{k,0}$, where  $\delta_{k,0}$ is the Kronecker delta.

For example, for $s=2$ and $n=1$ we have
  \[
  M_{0}^{(2)}(1)=1, \ \ M_{1}^{(2)}(1)=x_{1}, \ \ M_{2}^{(2)}(1)=0, \ \ M_{3}^{(2)}(1)=x^{3}_{1}
 \ \  M_{4}^{(2)}(1)=x^{4}_{1},\ \ M_{5}^{(2)}(1)=0.
  \]
  For $s=2=n$ we have 
   \[
  M_{0}^{(2)}(2)=1, \ \  M_{1}^{(2)}(2)=x_{1}+x_{2}, \ \  M_{2}^{(2)}(2)=x_{1}x_{2}, \ \  M_{3}^{(2)}(2)=x^{3}_{1}+
  x^{3}_{2}.  \]

From the definition of  $M_{k}^{(s)}(n)$ we have the following theorem.
\begin{theorem}\label{teo1} Let $s$ and $n$ be positive integers. Then
\begin{equation}\label{egf}
\sum_{k\geq 0}M_{k}^{(s)}(n)t^{k}=\prod_{i=1}^{n}\frac{1+x_{i}t}{1-\left(
x_{i}t\right) ^{s+1}}.
\end{equation}
\end{theorem}

Moreover, the modular symmetric function also satisfies the following recurrence relations.
\begin{theorem}\label{teo1a} Let $s$ and $n$ be positive integers. Then
\begin{align}
&M_{k}^{(s)}(n)=\sum_{\substack{ 0\leq j\leq k  \\ j\equiv (0,1)\mod(s+1) }}x_{n}^{j}M_{k-j}^{(s)}(n-1),\label{err2}\\
&M_{k}^{(s)}(n)=x_{n}^{s+1}M_{k-s-1}^{(s)}(n)+x_{n}M_{k-1}^{(s)}(n-1)+M_{k}^{(s)}(n-1),\label{err1}
\end{align}
for $k\geq s+1$.
\end{theorem}
\begin{proof}
From \eqref{emf} and the Theorem \ref{teo1} we have
\begin{eqnarray*}
\sum_{k\geq 0}M_{k}^{(s)}(n)t^{k} &=&\prod_{i=1}^{n}\frac{1+x_{i}t}{%
1-\left( x_{i}t\right) ^{s+1}} =\prod_{i=1}^{n}\left( (1+x_{i}t)\sum_{j\geq 0}\left(
x_{i}t\right) ^{(s+1)j}\right) \\
&=&\prod_{i=1}^{n}\left( \sum_{j\geq 0}\left( x_{i}t\right)
^{(s+1)j}+\sum_{j\geq 0}\left( x_{i}t\right) ^{(s+1)j+1}\right) \\
&=&\prod_{i=1}^{n}\left( \sum_{\ell\equiv (0,1)\mod(s+1)}\left(
x_{i}t\right) ^{\ell}\right) \\
&=&\sum_{j\equiv (0,1)\mod(s+1)}\left( x_{n}t\right)
^{j}\prod_{i=1}^{n-1}\left( \sum_{\ell\equiv (0,1)\mod%
(s+1)}\left( x_{i}t\right) ^{\ell}\right) \\
&=&\sum_{j\equiv (0,1)\mod(s+1)}\left( x_{n}t\right)
^{j}\sum_{\ell=0}^{\infty }M_{\ell}^{(s)}(n-1)t^{\ell} \\
&=&\sum_{k\geq 0}t^{k}\sum_{\substack{ j+\ell=k  \\ j\equiv (0,1)\mod%
(s+1) }}x_{n}^{j}M_{\ell}^{(s)}(n-1) \\
&=&\sum_{k\geq 0}t^{k}\sum_{\substack{ 0\leq j\leq k  \\ j\equiv (0,1)\mod(s+1)}}x_{n}^{j}M_{k-j}^{(s)}(n-1).
\end{eqnarray*}%
By comparing the $k$-th coefficient we obtain \eqref{err2}. The relation \eqref{err1} follows in a similar manner.\end{proof}
Notice that from \eqref{err2} we have the equality
$h_{k}(n)=\sum_{j=0}^{k}x_{n}^{j}h_{k-j}(n-1)$.

\section{Combinatorial interpretation}
The goal of this section is to present a combinatorial interpretation for the modular symmetric functions by means of weighted lattice paths in the plane $\Z\times \Z$.   A \emph{lattice path} $\Gamma$ in the lattice plane $\Z\times \Z$, with steps in a given set $S\subset\Z^2$, is a concatenation of directed steps of $S$, that is $\Gamma=s_1s_2\cdots s_\ell$, where $s_i\in S$, for each $1\leq i \leq \ell$. Let $\Pa_{n,k}$ denote the set of lattice paths from the point $(0,0)$ to the point $(k,n-1)$,  with step set $S=\{ H=(1,0), V=(0,1)\}$, such that the horizontal steps are labelled with the weight $x_i$, where $i-1$ is the level of the step.   Let $\Pa_{n,k}^{(s)}$ denote the weighted lattice path in $\Pa_{n,k}$ such that the number of horizontal steps in each level are congruent to $0$ or $1$ modulo $s+1$. Given a weighted path  $\Gamma$ in  $\Pa_{n,k}^{(s)}$, we denote by $\omega(\Gamma)$  the weight associated to the path $\Gamma$. For example, in Figure \ref{Fig1} we show a lattice path in $\Pa_{6,12}^{(2)}$ of  weight $x_2^6x_4x_5x_6^4$.
\begin{figure}[H]
\centering
\includegraphics[scale=0.7]{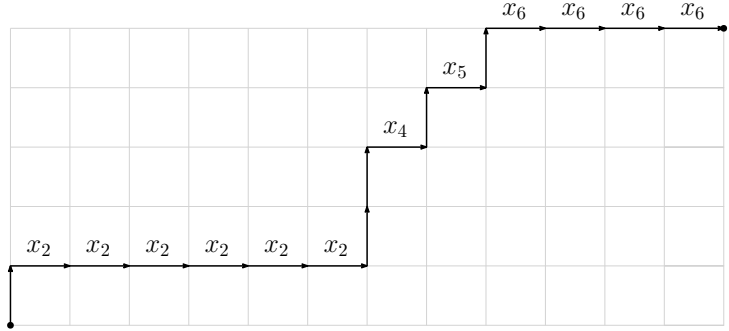}
\caption{Weighted lattice path in $\Pa_{6,12}^{(2)}$.} \label{Fig1}
\end{figure}

From \eqref{emf} we obtain the  following combinatorial interpretation.
\begin{theorem} \label{T4.4}
	Let $k,n$ and $s$ be  positive integers and let $x_{1},x_{2},\ldots,x_{n}$ be independent variables. Then
\[
M_{k}^{(s)}\left( x_{1},x_2,\ldots ,x_{n}\right) =\sum_{\Gamma \in \Pa_{n,k}^{(s)}}\omega(\Gamma).\]
\end{theorem}

Figure \ref{fg} shows the weighted lattice path interpretation for
$M_{3}^{(2)}(x_{1},x_{2},x_{3})=x_{1}^{3}+x_{2}^{3}+x_{3}^{3}+x_1x_2x_3$.

\begin{figure}[h]
	\begin{center}
		\begin{tikzpicture}
		\draw[step=0.7cm,color=lightgray] (0,0) grid(2.1,1.4);
		\draw [line width=0.7pt](0,0) -- (0.7,0)--(1.4,0) -- (2.1,0)--(2.1,1.4);	
		\fill[black] (0,0) circle(1.5pt) ;
		\fill[black] (2.1,1.4) circle(1.5pt) ;	
		\fill[] (0.7,0) circle (0.9pt);\fill[] (1.4,0) circle (0.9pt);\fill[] (2.1,0) circle (0.9pt);
		\fill[] (0,0.7) circle (0.9pt);\fill[] (0.7,0.7) circle (0.9pt);\fill[] (1.4,0.7) circle (0.9pt);
		\fill[] (2.1,0.7) circle (0.9pt);
		\fill[] (0,1.4) circle (0.9pt);\fill[] (0.7,1.4) circle (0.9pt);\fill[] (1.4,1.4) circle (0.9pt);
		\node[rectangle] at (0.4,0.3) {$x_{1}$};
		\node[rectangle] at (1.1,0.3) {$x_{1}$};
		\node[rectangle] at (1.8,0.3) {$x_{1}$};
		\begin{scope}[xshift=1.0cm]
		\begin{scope}[xshift=2.2cm]
		\draw[step=0.7cm,color=lightgray] (0,0) grid(2.1,1.4);
		\draw [line width=0.7pt](0,0) -- (0,0.7)-- (0.7,0.7) -- (1.4,0.7) -- (1.4,0.7)--(2.1,0.7)--(2.1,1.4);
		
		\fill[black] (0,0) circle(1.5pt) ;
		\fill[black] (2.1,1.4) circle(1.5pt) ;
		
		\fill[] (0.7,0) circle (0.9pt);\fill[] (1.4,0) circle (0.9pt);\fill[] (2.1,0) circle (0.9pt);
		\fill[] (0,0.7) circle (0.9pt);\fill[] (0.7,0.7) circle (0.9pt);\fill[] (1.4,0.7) circle (0.9pt);
		\fill[] (2.1,0.7) circle (0.9pt);
		\fill[] (0,1.4) circle (0.9pt);\fill[] (0.7,1.4) circle (0.9pt);\fill[] (1.4,1.4) circle (0.9pt);
		\node[rectangle] at (0.4,0.9) {$x_2$};
		\node[rectangle] at (1.1,0.9) {$x_2$};
		\node[rectangle] at (1.75,0.9) {$x_2$};
		\begin{scope}[xshift=1.0cm]
		\begin{scope}[xshift=2.2cm]
		\draw[step=0.7cm,color=lightgray] (0,0) grid(2.1,1.4);
		\draw [line width=0.7pt](0,0) -- (0,1.4)-- (0.7,1.4) -- (1.4,1.4) -- (2.1,1.4);
	\fill[black] (0,0) circle(1.5pt) ;
		\fill[black] (2.1,1.4) circle(1.5pt) ;
		\fill[] (0.7,0) circle (0.9pt);\fill[] (1.4,0) circle (0.9pt);\fill[] (2.1,0) circle (0.9pt);
		\fill[] (0,0.7) circle (0.9pt);\fill[] (0.7,0.7) circle (0.9pt);\fill[] (1.4,0.7) circle (0.9pt);
		\fill[] (2.1,0.7) circle (0.9pt);
		\fill[] (0,1.4) circle (0.9pt);\fill[] (0.7,1.4) circle (0.9pt);\fill[] (1.4,1.4) circle (0.9pt);
		\node[rectangle] at (0.4,1.6) {$x_3$};
		\node[rectangle] at (1.1,1.6) {$x_3$};
		\node[rectangle] at (1.75,1.6) {$x_3$};
		\begin{scope}[xshift=1.0cm]
			\begin{scope}[xshift=2.2cm]
	\draw[step=0.7cm,color=lightgray] (0,0) grid(2.1,1.4);
		\draw [line width=0.7pt](0,0) -- (0.7,0)-- (0.7,0.7) -- (1.4,0.7) -- (1.4,1.4)--(2.1,1.4);
			\fill[black] (0,0) circle(1.5pt) ;
		\fill[black] (2.1,1.4) circle(1.5pt) ;
		\fill[] (0.7,0) circle (0.9pt);\fill[] (1.4,0) circle (0.9pt);\fill[] (2.1,0) circle (0.9pt);
		\fill[] (0,0.7) circle (0.9pt);\fill[] (0.7,0.7) circle (0.9pt);\fill[] (1.4,0.7) circle (0.9pt);
		\fill[] (2.1,0.7) circle (0.9pt);
		\fill[] (0,1.4) circle (0.9pt);\fill[] (0.7,1.4) circle (0.9pt);\fill[] (1.4,1.4) circle (0.9pt);
		\node[rectangle] at (0.4,0.2) {$x_1$};
		\node[rectangle] at (1.1,0.9) {$x_2$};
		\node[rectangle] at (1.75,1.6) {$x_3$};
		\end{scope}
		\end{scope}
		\end{scope}
		\end{scope}
		\end{scope}
		\end{scope}
		\end{tikzpicture}
	\end{center}
	\caption{The four paths associated to $M_{3}^{(2)}(x_1,x_2,x_3)$.}
	\label{fg}
\end{figure}
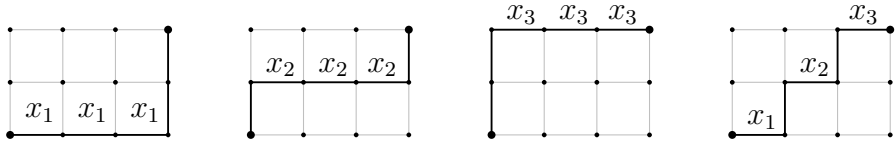

\subsection{Tiling interpretation}
In this section, we use weighted tilings  to give an additional combinatorial interpretation of the modular symmetric function.  We define a \emph{weighted tiling} as a tiling of a board of length $n$ ($n$-board) by gray and black squares, such that each black square received the weight $x_{m+1}$, where $m$ is  equal to the number of gray squares to the left of that black square in the tiling.
Let   $\T^{(s)}_{n,k}$  denote the set of weighted tilings of an $(n+k-1)$-board using exactly $k$ black squares and $n-1$ gray squares, such that the number of successive black squares is congruent to $0$ or $1$ modulo $s+1$. For a tiling $T$, we denote by $\omega(T)$ the weight of $T$.

For example, in Figure \ref{Fig2} we show a weighted tiling in $\T_{6,12}^{(2)}$  of weight $x_2^6x_4x_5x_6^4.$
\begin{figure}[H]
\centering
\includegraphics[scale=0.7]{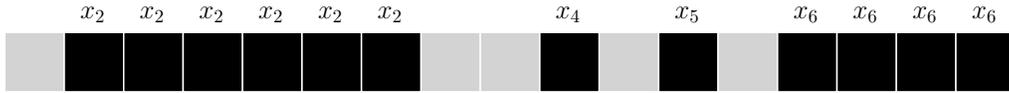}
\caption{Weighted tiling in $\T_{6,12}^{(2)}$.} \label{Fig2}
\end{figure}

There is a bijection between the sets $\Pa^{(s)}_{n,k}$ and  $\T^{(s)}_{n,k}$. Indeed, each vertical step $V$ is replaced by a gray square and each horizontal step is replaced by a black square. Since the bijection between lattice paths and tiling is weight-preserving, we obtain the following result.
\begin{theorem} \label{T4.5}
	Let $k,n$ and $s$ be  positive integers and let $x_{1},x_{2},\ldots,x_{n}$ be independent variables. Then
\[
M_{k}^{(s)}\left( x_{1},x_2,\ldots ,x_{n}\right) =\sum_{T\in \mathcal{T}^{(s)}_{n,k}}\omega(T).\]
\end{theorem}

Figure \ref{Fig5} shows the tiling interpretation for $M_{3}^{(2)}(x_1,x_2,x_3)=x_{1}^{3}+x_{2}^{3}+x_{3}^{3}+x_1x_2x_3$.

\begin{figure}[h]
	
	\begin{center}
		\begin{tikzpicture}
		
		\draw[step=0.7cm,color=black!30] (0,0) grid(3.5,0.7);
		\fill[fill=black,draw=black] (0.05,0.05) rectangle (0.65,0.65);
		\fill[fill=black,draw=black] (0.75,0.05) rectangle (1.35,0.65);
		\fill[fill=black,draw=black] (1.45,0.05) rectangle (2.05,0.65);
		\fill[fill=gray!50,draw=gray!50] (2.15,0.05) rectangle (2.75,0.65);
		\fill[fill=gray!50,draw=gray!50] (2.85,0.05) rectangle (3.45,0.65);
		
		\begin{scope}[xshift=1.5cm]
		\begin{scope}[xshift=2.7cm]
		
		\draw[step=0.7cm,color=black!30] (0,0) grid(3.5,0.7);
		\fill[fill=gray!50,draw=gray!50] (0.05,0.05) rectangle (0.65,0.65);
		\fill[fill=black,draw=black] (0.75,0.05) rectangle (1.35,0.65);
		\fill[fill=black,draw=black] (1.45,0.05) rectangle (2.05,0.65);
		\fill[fill=black,draw=black] (2.15,0.05) rectangle (2.75,0.65);
		\fill[fill=gray!50,draw=gray!50] (2.85,0.05) rectangle (3.45,0.65);
		
        \end{scope}
		\end{scope}

		\begin{scope}[yshift=-0.5cm]
		\begin{scope}[yshift=-1.0cm]
		
		\draw[step=0.7cm,color=black!30] (0,0) grid(3.5,0.7);
		\fill[fill=gray!50,draw=gray!50] (0.05,0.05) rectangle (0.65,0.65);
		\fill[fill=gray!50,draw=gray!50] (0.75,0.05) rectangle (1.35,0.65);
		\fill[fill=black,draw=black] (1.45,0.05) rectangle (2.05,0.65);
		\fill[fill=black,draw=black] (2.15,0.05) rectangle (2.75,0.65);
		\fill[fill=black,draw=black] (2.85,0.05) rectangle (3.45,0.65);
		
		   		\begin{scope}[xshift=1.5cm]
		   \begin{scope}[xshift=2.7cm]
		
		   \draw[step=0.7cm,color=black!30] (0,0) grid(3.5,0.7);
		   \fill[fill=black,draw=black] (0.05,0.05) rectangle (0.65,0.65);
		   \fill[fill=gray!50,draw=gray!50] (0.75,0.05) rectangle (1.35,0.65);
		   \fill[fill=black,draw=black] (1.45,0.05) rectangle (2.05,0.65);
		   \fill[fill=gray!50,draw=gray!50] (2.15,0.05) rectangle (2.75,0.65);
		   \fill[fill=black,draw=black] (2.85,0.05) rectangle (3.45,0.65);
		
		   \end{scope}
		   \end{scope}
		
        \end{scope}
        \end{scope}
  		\end{tikzpicture}
	\end{center}
\caption{The four tilings associated to $M_3^{(2)}(x_1,x_2,x_3)$.}\label{Fig5}
\end{figure}
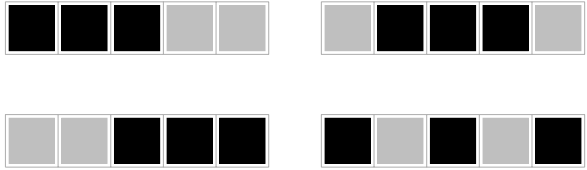
In Theorem \ref{teo32} we  give a combinatorial expression for the sequence $M_{k}^{(s)}(\underbrace{1,1,\ldots ,1}_{n \text{\ times}})$.

\begin{theorem}\label{teo32}
For $n, k \geq 0, s\geq 1$, we have
\begin{equation*}
M_{k}^{(s)}(\underbrace{1,1,\ldots ,1}_{n \ \text{\emph{times}}})=\sum_{j=0}^{\left\lfloor \frac{k}{s+1}%
\right\rfloor }\binom{n}{k-j(s+1)}\binom{j+n-1}{n-1}.
\end{equation*}
\end{theorem}
\begin{proof}
From the combinatorial interpretation $M_{k}^{(s)}(\overbrace{1,1,\ldots ,1}^{n \text{\ times}})$ counts the number of weighted tilings of a $(n+k-1)$-board using exactly $k$ black squares and $n-1$ gray squares, such that the number of successive black squares is congruent to $0$ or $1$ modulo $s+1$.  On the other hand, let $j$ be the number of successive black squares multiples of $s+1$. Notice that $0\leq j \leq \lfloor k/(s+1) \rfloor$. Then there are $j+n-1$ gray and black blocks tiles.  Such a tiling with $j+n-1$ tiles, exactly $j$ of which are black blocks of size  congruent to $0$ module $s+1$ is
$\binom{j+n-1}{j}.$ The remaining $k-j(s+1)$ black squares can be inserted before to each gray square or to the end of the tiling. Since there are $n-1$ gray squares we have $\binom{n}{k-j(s+1)}$ ways to insert the black squares. Hence there are $\binom{n}{k-j(s+1)}\binom{j+n-1}{j}$ tilings altogether. Summing over all $j$ gives the total number of weighted tiling in $\mathcal{T}^{(s)}_{n,k}$, which implies the identity. \qedhere
\end{proof}

Notice that we can also give an algebraic proof for the above result. Indeed, from the generating function given in Theorem \ref{teo1} we have
\begin{align*}
\sum_{k\geq 0}M_{k}^{(s)}(1,1,\ldots ,1)t^{k} &=\frac{(1+t)^{n}}{\left(
1-t^{s+1}\right) ^{n}}=\sum_{j=0}^{n}\binom{n}{j}%
t^{j}\sum_{\ell \geq 0}\binom{\ell+n-1}{n-1}t^{\ell(s+1)} \\
&=\sum_{j=0}^{n}\sum_{\ell\geq 0}\binom{n}{j}\binom{\ell+n-1}{n-1}%
t^{j+\ell(s+1)} = \sum_{k\geq 0}t^{k}\sum_{j+\ell(s+1)=k}\binom{n}{j}\binom{%
\ell+n-1}{n-1} \\
&=\sum_{k\geq 0}t^{k}\sum_{\ell=0}^{\left\lfloor \frac{k}{s+1}%
\right\rfloor }\binom{n}{k-\ell(s+1)}\binom{\ell+n-1}{n-1}.
\end{align*}
By comparing the $k$-th coefficient we obtain the desired result.

\section{Modular $s$-Stirling numbers}
The \textit{Stirling numbers of the second kind} ${n \brace k}$  can be determined by the recurrence relation
${n \brace k}={n-1 \brace k-1}+k{n-1 \brace k}$, with the initial conditions ${0 \brace 0}=1$ and ${n \brace 0}={0 \brace n}=0$ for $n\geq 1$. It is well-known  that   the ${n \brace k}$ are determined by the identities
$x^n=\sum_{k=0}^n{n \brace k}x^{\underline{k}}, n \geq 0$, 
where $x^{\underline{n}}=x(x-1)\cdots (x-(n-1))$ for $n\geq  1$ and $x^{\underline{0}}=1$
or equivalently by the generating function
\begin{align}
\sum_{n \geq k}{n \brace k}x^n&=\frac{x^k}{(1-x)(1-2x)\cdots (1-kx)}, \quad k\geq 0. \label{ogf2}
\end{align}
Using \eqref{ogf2}, it is not difficult to show that the Stirling numbers of the second kind  are the specialization of the complete symmetric function given by
 \begin{align}
{n+k \brace n}&=h_k(1, 2, \dots, n)=\sum _{a_1+\cdots +a_n=k}1^{a_1}\cdots n^{a_n}.\label{s2}
 \end{align}
The Equation \eqref{s2} can be interpreted by considering the following algorithm:

\begin{algorithm}[H]\label{Alg}
\begin{enumerate}[(1) ]
\item Start with the partition of $[1]$ given by $\{1\}$. 
\item Take every integer from $2$ to $1+a_1$ and put it in the block of $1,$ so you end up having $\{1, 2, \dots ,1+a_1\}.$ 
\item Then you place $2+a_1$ in a new block and for every integer in between $3+a_1$ and $3+a_1+a_2$ you have $2$ options, either you place this number in the first block or in the second one. You place $3+a_1+a_2$ in a new block and so now you will have $3$ options. 
\item You keep doing this until you have placed $n+k$ elements.
\end{enumerate}
\caption{Interpretation of \eqref{s2}.}
\label{ReverseFunction}
\end{algorithm}

For example, for $n=3$ and $k=5$ the term $1^{2}2^{1}3^{2}$ corresponds to a partition that looks like
$\{1,2,3,*\}, \ \{4,*\}, \ \{6,*\}$,  where $5$ can go in either of the first $2$ blocks and $7,8$ can go in any block (there are $3$ of them). Giving a total of $2^1\cdot 3^2$ options.

Notice then that the $a_i$ integers have a direct relationship with the minimal elements in each of the blocks of a partition. To see this, consider the following construction: let $\Pi(n,k)$ denote the set of partitions of $[n]$ having $k$ blocks. Suppose $\pi \in \Pi(n,k)$ is represented as $\pi=B_1/B_2/\cdots/B_k$, where $B_i$ denotes the $i$-th block, with $\min(B_1)<\min(B_2)<\cdots <\min(B_k)$. Call $m_i=\min(B_i)$ and define the vector of consecutive differences by
 $$d(\pi):=(d_1,\dots ,d_{k})=(m_2-m_1-1,m_3-m_2-1,\dots ,m_k-m_{k-1}-1,n-m_{k}).$$
In the example above, notice that $d_i$ corresponds to $a_i$ because these are exactly the number of elements that we have to place in $i$ blocks and so there are a total of $i^{d_i}$ ways to do this. Notice, further, that since $m_1=1,$ we have that $d_1+\cdots +d_k=n-k$. If we impose the modularity conditions on the $d_i$'s, we get the modular symmetric function defined in Equation \eqref{emf}.

Notice that \eqref{s2} can be written as ${n \brace k}=h_{n-k}(1, 2, \dots, k).$ From this  last equation and the combinatorial motivation of $d(\pi),$ we introduce a new kind of Stirling numbers.  For all integer $n\geq 0$ and all $k$ with $0\leq k \leq n$  \emph{the modular $s$-Stirling numbers of the second kind}, denoted by ${n \brace k}^{(s)}$, are defined by the expression
 \begin{align}\label{s3}
 {n \brace k}^{(s)}=M_{n-k}^{(s)}(1,2, \dots, k).
  \end{align}
It is clear that for $s=1$ we recover the Stirling numbers of the second kind, that is,  ${n \brace k}^{(1)}={n \brace k}$. From Theorem \ref{teo1a} we have the following recurrence relation:
\begin{equation}\label{ems}
{n \brace k}^{(s)}={n-1 \brace k-1}^{(s)}+k{n-2\brace k-1}^{(s)}+k^{s+1}{n-s-1 \brace k}^{(s)},
\end{equation}
with the initial conditions ${n \brace 0}^{(s)}=\delta _{0,n}$ and ${0 \brace k}^{(s)}=\delta _{k,0}$. Moreover,  we have the following generating function. For positive integers $n$ and $s$, we have
\begin{equation*}
\sum_{n\geq k}{n \brace k}^{(s)}x^{n-k}=\prod_{r=1}^{k}\frac{1+rx^{s}}{1-(rx)^{s+1}}.
\end{equation*}

In Theorem \ref{comin} we give a combinatorial interpretation for the modular $s$-Stirling numbers.
\begin{theorem}\label{comin}
The number of set partitions $\pi$ in $\Pi(n,k)$, such that  the entries in the vector $d(\pi)$ satisfies $d_i\equiv 0,1 \pmod{s+1}$ for each $1\leq i \leq k$ is given by the modular $s$-Stirling numbers ${n\brace k}^{(s)}$.
\end{theorem}
\begin{proof}
By imposing the modularity conditions on the vector of consecutive differences given above and applying the Algorithm described on Page \pageref{Alg}, the theorem follows.
\end{proof}

For example, ${5\brace 2}^{(2)}=9$ corresponding to the set partitions
\begin{align*}
&1234/5, \quad 1345/2, \quad 134/25, \quad 135/24, \quad 13/245,\\
&145/23, \quad 14/235, \quad 15/234, \quad 1/2345.
\end{align*}
If you restrict the difference vector to have elements of the form $d_i\equiv 0 \pmod{s+1},$ then the number of such partitions is given by $h_{\lfloor \frac{n-k}{s+1}\rfloor}(1^{s+1},\dots ,k^{s+1}).$

On the other hand, the  (unsigned) \emph{Stirling numbers of the first kind},
 satisfies the recurrence relation ${n \brack k}=(n-1){n-1 \brack k} + {n-1 \brack k-1}$, with the initial conditions ${0 \brack 0}=1$ and ${n \brack 0}={0 \brack n}=0$ for $n\geq 1$. This sequence can also  defined as the connection constants in the polynomial identity
\begin{equation}\label{w1prop1}
x(x+1)\cdots(x+(n-1))=\sum_{k=0}^n{n \brack k}x^k.
\end{equation}

\begin{theorem}\label{ps1}
Let $n$ and $k$ be non negative integers  and $s>0$. If $r$ is the remainder of $k$ when divided by $s+1$, then the following equation holds
\begin{align*}
{n+k \brace n}^{(s)}=\sum _{i = 0}^{\min \{\lfloor \frac{n-r}{s+1}\rfloor ,\lfloor \frac{k}{s+1}\rfloor \}}h_{\lfloor \frac{k}{s+1}\rfloor -i}(1^{s+1},\dots ,n^{s+1}){n+1\brack n+1-r-i(s+1)}.
\end{align*}
\end{theorem}
\begin{proof}
From Equation \eqref{emf}  consider the subset of the variables $\{a_i\}_{i\in [n]}$ such that $a_j\equiv 1\pmod {s+1}.$ Call $J\subseteq [n]$ the set of their subindices in such a way that $j\in J$ if and only if $a_j\equiv 1\pmod {s+1}.$ Notice that for each $j\in J$, one must have that $a_j=b_j(s+1)+1$ for some $b_j\in \mathbb{N}.$ Therefore we have \begin{align*}
   {n+k \brace n}^{(s)}&= M^{(s)}_k(1,\dots ,n)=\sum _{\substack{
    J\subseteq [n]\\|J|\equiv r \pmod{s+1}}}\,\sum _{\substack{a_1+\cdots +a_n=k\\a_j\equiv 1 \pmod{s+1}\\\text{for }j\in J}}1^{a_1}\cdots n^{a_n}\\
    &=\sum _{\substack{
    J\subseteq [n]\\ |J|\equiv r \pmod{s+1}}}\prod _{j\in J}j\sum _{(b_1+\cdots +b_n+\lfloor \frac{|J|}{s+1} \rfloor)+\frac{r}{s+1}=\frac{k}{s+1}}1^{(s+1)b_1}\cdots n^{(s+1)b_n}\\
    &=\sum _{\substack{
    J\subseteq [n]\\ |J|\equiv r \pmod{s+1}}}\prod _{j\in J}j\sum _{b_1+\cdots +b_n=\lfloor \frac{k}{s+1}\rfloor -\lfloor \frac{|J|}{s+1} \rfloor}1^{(s+1)b_1}\cdots n^{(s+1)b_n}\\
    &=\sum _{i=0}^{\lfloor \frac{n-r}{s+1}\rfloor }\left (\sum _{\substack{|J|=i(s+1)+r\\
    J\subseteq [n]}}\prod _{j\in J}j\right )h_{\lfloor \frac{k}{s+1}\rfloor -i}(1^{s+1},\dots ,n^{s+1}).
\end{align*}
From \eqref{w1prop1} the equality follows.
\end{proof}
 For example, for $n=4, k=8,$ and $s=3$ we have ${4+8 \brace 4}^{(s)}=107331$. On the other hand,
\begin{multline*}
\sum _{i = 0}^{1}h_{2-i}(1^{4},2^4, 3^4 ,4^{4}){5\brack 5-4i}\\=(1^42^4+1^43^4+2^43^4+1^44^4+2^44^4+3^44^4+1^8+2^8+3^8+4^8)\cdot 1 \\+ (1^4
+2^4+
3^4+4^4)\cdot 24=107331.
\end{multline*}

From Theorem \ref{ps1} and by the little Fermat's theorem $a^p\equiv a \pmod p$ with $a\in\mathbb{Z}$, we conclude the following interesting congruence.
\begin{corollary}
Let $n$ and $k$ be non negative integers  and $p$ a prime number. If  $r$ is the remainder of $k$ when divided by $p$, then the following congruence holds
\begin{equation*}
{n+k \brace n}^{(p-1)}\equiv \sum _{i=0}^{\min \{\lfloor \frac{n-r}{p}\rfloor ,\lfloor \frac{k}{p}\rfloor \}}{n+\lfloor \frac{k}{p}\rfloor -i \brace n}{n+1\brack n+1-(r+i\cdot p)}\pmod p.
\end{equation*}
\end{corollary}

\subsection{The $\ell$-modular symmetric function}

Given $0\leq \ell <s+1 ,$ we can define the \emph{$\ell$-modular symmetric function} by
$$M_k^{(s,\ell)}(x_1,\dots ,x_n):=\sum _{\substack{j_1+\cdots +j_n=k\\ j_1,\cdots ,j_n \equiv 0,\ell \pmod {s+1}}}x_1^{j_1}\cdots x_n^{j_n}.$$
Using the definition above, we can extend the Theorem \ref{ps1} using the \emph{Stirling numbers of the first kind with higher level}, defined in \cite{Komatsu, Komatsu2}. Moreover, some applications of the Stirling numbers of higher level in special polynomials can be found in \cite{KomPit, Kom2}. 

Let $\Sg_n$ denote the set of permutations of the set  $[n]$. We will assume that permutations are expressed in \emph{standard cycle form}, i.e., minimal elements first within each cycle,
with cycles arranged left-to-right in ascending order of minimal elements. If $n, k \geq 0$, then let  $\Sg_{(n,k)}$ denote the set of permutations of $\Sg_n$  having exactly $k$ cycles. It is clear that $\Sg_n=\cup_{k=0}^n\Sg_{(n,k)}$ and  $|\Sg_{(n,k)}|={n \brack k}$. Given a permutation $\sigma$ in $\Sg_n$, let $\min(\sigma)$ denote the set of the minimal elements in each cycle of $\sigma$.  For example, if $\sigma=(1\, 4 \, 5)(2\, 3)(6)(7 \, 9)(8)$, then we have that $\min(\sigma)=\{1, 2, 6, 7, 8\}$.

 Given a positive integer $s$, let $\sttf2{n}{k}_s$ denote the number of ordered $s$-tuples $(\sigma_1, \sigma_2, \dots, \sigma_s)\in \Sg_{(n,k)}\times \Sg_{(n,k)} \times \cdots \times \Sg_{(n,k)}=\Sg_{(n,k)}^s$, such that
$\min(\sigma_1)=\min(\sigma_2)=\cdots=\min(\sigma_s).$

The sequence $\sttf2{n}{k}_s$ satisfies the following recurrence relation
\begin{equation}
\sttf2{n}{k}_s=\sttf2{n-1}{k-1}_s+(n-1)^s\sttf2{n-1}{k}_s, \label{generalrec2}
\end{equation}
with the initial conditions $\sttf2{0}{0}_s=1$ and $\sttf2{n}{0}_s=\sttf2{0}{n}_s=0$ hold for $n\geq 1$.

Given integers $n\geq 0$ and $s\geq 1$, let $\Omega_{n,s}(x)$ denote the polynomials
$$\Omega_{n,s}(x):=x(x+1^s)(x+2^s)\cdots(x+(n-1)^s), \quad  \text{with} \quad  \Omega_{0,s}(x)=1.$$
 The Stirling numbers of the first kind with higher level are the connection constants between the polynomials $(\Omega_{n,s}(x))_{n\geq 0}$ and the canonical basis $(x^n)_{n\geq 0}$. Indeed, if $n\geq 0$, then
\begin{equation}\label{sH1}
\Omega_{n,s}(x)=x(x+1^s)(x+2^s)\cdots(x+(n-1)^s)=\sum_{k=0}^n\sttf2{n}{k}_s x^k\,.
\end{equation}
From a similar argument as in Theorem \ref{ps1} and from \eqref{sH1} we can obtain the following  theorem.

\begin{theorem}
Let $n$ and $k$ be non negative integers  and $s+1>\ell \geq 0$, such that $\gcd(\ell, s+1)=1$. If $r$ is the remainder of $k \ell ^{-1}$ when divided by $s+1$, then the following equation holds
\begin{multline*}
    M^{(s,\ell)}_k(1,\dots ,n)\\=\sum _{i = 0}^{\min \{\lfloor \frac{n-r}{s+1}\rfloor ,\lfloor \frac{k}{s+1}\rfloor - \lfloor \frac{r \ell}{s+1}\rfloor \}}h_{\lfloor \frac{k}{s+1}\rfloor -\lfloor \frac{r \ell}{s+1}\rfloor -i \ell}(1^{s+1},\dots ,n^{s+1})\left [{n+1\brack n+1-r-i (s+1)}\right ]_\ell.
\end{multline*}
\end{theorem}

\section{The $s$-elementary symmetric function}
The \emph{$s$-elementary symmetric polynomial} is defined by the expression
\begin{align}\label{defele}
E_k^{(s)}(n)=\sum _{\substack{a_1+\cdots +a_n=k\\a_i\leq s}}x_1^{a_1}\cdots x_n^{a_n}.
\end{align}
An equivalent definition of this symmetric polynomial already exists in a paper by Bazeniar et al. \cite{BAZ}. For further properties of this symmetric function see \cite{MMerc}.

\begin{theorem}
If $s\equiv 1 \pmod 2,$ then for every positive integers $n$ and $k$ the following identity holds
$$\sum _{i=0}^k(-1)^iE_i^{(s)}(n)\cdot M^{(s)}_{k-i}(n)=0.$$
\end{theorem}
\begin{proof}
The inverse of the generating function in Theorem \ref{teo1}  is given by
\begin{align*}
  \prod _{i=1}^n\frac{1-(x_it)^{s+1}}{1+x_it}&=\prod _{i=1}^n\frac{1-(-x_it)^{s+1}}{1-(-x_it)}=\prod _{i=1}^n\left (1-x_it+(-x_it)^2+\cdots +(-x_it)^{s}\right).
\end{align*}
In each product, we can create any number in between $1$ and $s$. Hence
$$\sum _{k=0}^{n\cdot s}E_k^{(s)}(n)(-t)^k=\prod _{i=1}^n\frac{1-(x_it)^{s+1}}{1+x_it},$$
and the desired identity follows.
\end{proof}

We can express the modular symmetric function $M_k^{(s)}$  as convolutions involving the complete and elementary symmetric functions as follows.
\begin{theorem}\label{T3.3}
	Let $k$, $n$ and $s$ be positive integers and let $x_1,x_2,\ldots,x_n$ be independent
	variables. Then
	$$M_{k}^{(s)}(x_{1},x_{2},\ldots,x_{n})
	=\sum_{j=0}^{\lfloor k/s+1 \rfloor}  h_{j}(x_{1}^{s+1},x_{2}^{s+1},\ldots,x_{n}^{s+1})
	e_{k-(s+1)j}(x_{1},x_{2},\ldots,x_{n}).$$
\end{theorem}
\begin{proof}
\allowdisplaybreaks{
According to \eqref{egf}, we have
\begin{align*}
& \sum_{k=0}^\infty M_k^{(s)} (x_{1},x_{2},\ldots,x_{n}) t^k\\
& \qquad=\left(\prod_{i=1}^{n}\frac{1}{1-(x_{i}t)^{s+1}}\right)\left( \prod_{i=1}^{n} (1+x_{i}t)\right)\\
& \qquad=\left(\sum_{j=0}^\infty  h_{j}(x_{1}^{s+1},x_{2}^{s+1},\ldots,x_{n}^{s+1})(t)^{(s+1)j} \right)\left(\sum_{j=0}^\infty e_{j}(x_{1},x_{2},\ldots,x_{n})t^{j}\right)\\
& \qquad=\sum_{k=0}^\infty \left(\sum_{j=0}^{\lfloor k/s+1 \rfloor}h_{j}(x_{1}^{s+1},x_{2}^{s+1},\ldots,x_{n}^{s+1})
e_{k-(s+1)j}(x_{1},x_{2},\ldots,x_{n})\right)t^{k}.
\end{align*}
As required.}
\end{proof}
Inspired by Theorem \ref{T3.3}, we provide the following generalization.
\begin{theorem}\label{T3.4}
Let $k$, $n$ and $s$ be three positive integers and let $x_1,x_2,\ldots,x_n$ be independent
	variables. Then
	$$
	h_k(x_1^s,x_2^s,\ldots,x_n^s)
	= \sum_{j=0}^{k(s+1)} (-1)^{j} h_j(x_1,x_2,\ldots,x_n) M_{k(s+1)-j}^{(s)}(x_1,x_2,\ldots,x_n)
	$$
    and
    $$
	e_k(x_1,x_2,\ldots,x_n)
	= \sum_{j=0}^{\lfloor k/s+1 \rfloor} (-1)^{j} e_j(x_1,x_2,\ldots,x_n) M_{k-j(s+1)}^{(s)}(x_1,x_2,\ldots,x_n).
	$$
	If $k$ is not congruent to  $0$ modulo $s+1$, then
	$$
	\sum_{j=0}^k (-1)^j h_j(x_1,x_2,\ldots,x_n) M_{k-j}^{(s)}(x_1,x_2,\ldots,x_n) = 0.
	$$
\end{theorem}

\begin{proof}
The relation given in Theorem \ref{teo1} can be rewritten as
	$$\prod_{i=1}^n \frac{1}{1+x_it} \sum_{k=0}^\infty M_k^{(s)}(x_1,x_2,\ldots,x_n) t^k
	= \prod_{i=1}^n \frac{1}{1-(x_it)^{s+1}}$$
	or
	$$\prod_{i=1}^n (1-(x_i t)^{s+1}) \sum_{k=0}^\infty M_k^{(s)}(x_1,x_2,\ldots,x_n) t^k
	= \prod_{i=1}^n \big(1+(x_it)\big).$$
	Thus we deduce that
	\begin{align*}
	& \sum_{k=0}^\infty h_k(x_1^{s+1},x_2^{s+1},\ldots,x_n^{s+1}) t^{k(s+1)}\\
	& \qquad = \left( \sum_{k=0}^\infty (-1)^k h_k(x_1,x_2,\ldots,x_n) t^k \right)
	\left( \sum_{k=0}^\infty M_k^{(s)}(x_1,x_2,\ldots,x_n) t^k \right)
	\end{align*}
	and
	\begin{align*}
	& \sum_{k=0}^\infty e_k(x_1,x_2,\ldots,x_n) t^{k}\\
	& \qquad = \left( \sum_{k=0}^\infty (-1)^k e_k(x_1^{s+1},x_2^{s+1},\ldots,x_n^{s+1}) t^{k(s+1)} \right)
	\left( \sum_{k=0}^\infty M_k^{(s)}(x_1,x_2,\ldots,x_n) t^k \right).
	\end{align*}
	The proof follows easily by comparing the coefficients of $t^{ks}$ on both sides of these equations.	
\end{proof}

The following result allows us to express a convolution of the modular symmetric function $M_k^{(s)}$  as convolutions involving the complete and elementary symmetric functions.

\begin{theorem}
Let $k$, $n$ and $s$ be three positive integers and let $x_1,x_2,\ldots,x_n$ be independent
	variables. Then
$$
\sum_{j=0}^{k}e_j(x_1,x_2,\ldots,x_n)h_{k-j}(x_1,x_2,\ldots,x_n)=\sum_{j=0}^{k}M^{(s)}_j(x_1,x_2,\ldots,x_n)E^{(s)}_{k-j}(x_1,x_2,\ldots,x_n).
$$
 \end{theorem}
We can, now define the modular $s$-Stirling numbers of the first kind by the following equality
$${n+1\brack k+1}^{(s)}=(n!)^{s}E_k^{(s)}\left(1,\frac{1}{2},\dots ,\frac{1}{n}\right).$$
These numbers were introduced independently by Ahmia et al. \cite{ABA}. They interpreted ${n\brack k}^{(s)}$ as the number of $s$-tuple
permutations of $[n]$ having together $k$ cycles. Inspired by the combinatorial interpretation given in Theorem \ref{comin} for the modular $s$-Stirling numbers of the second kind, we
give in the following  theorem another  combinatorial interpretation of ${n\brack k}^{(s)}$.
\begin{theorem}\label{teosti1}
Let $n$ and $k$ be non negative integers  and $s>0$. The $s$-modular Stirling numbers of the first kind ${n+1\brack k+1}^{(s)}$ count the number of set partitions $\pi \in \Pi (n(s+1)-k,n)$ such that $d(\pi)=(d_1,\dots ,d_n)$ has the property that $d_i\leq s$ for every $1\leq i \leq n$.
\end{theorem}
\begin{proof}
From \eqref{defele} we have the equality
$$(n!)^sE_k^{(s)}\left(1,\frac{1}{2},\dots ,\frac{1}{n}\right)=E_{n\cdot s-k}^{(s)}(1,\dots ,n).$$
Using the same argument as in Theorem \ref{comin}, we obtain the desired result.
\end{proof}

From Theorem \ref{teosti1} (taking $s=1$) we obtain a probably new combinatorial interpretation for the Stirling numbers of the first kind in terms of set partitions. Indeed, ${n+1 \brack k+1}$ enumerates the set partitions in  $\Pi (2n-k,n)$, such that the vector  $d(\pi)=(d_1,\dots ,d_n)$ has the property that $d_i\leq 1$ for every $1\leq i \leq n$. For example, ${3+1 \brack 1+1}=11$, the partitions being
\begin{align*}
&1/23/45, \quad 1/235/4, \quad  12/3/45, \quad 13/2/45, \quad 12/34/5, \quad 12/35/4, \\
&135/2/4, \quad 15/23/4, \quad 124/3/5, \quad 125/3/4, \quad 13/25/4.
\end{align*}

The modular $s$-Stirling numbers of the first kind satisfy the following recurrence  relation
\begin{align}
{n\brack k}^{(s)}=\sum _{\ell =0}^s{n-1\brack k-(s-\ell)}^{(s)}\cdot (n-1)^\ell,
    \label{recst1mod}
\end{align}
where ${0\brack 1-s}^{(s)}=1,{n\brack k}^{(s)}=0$ if $k<1-s$. See also, Ahmia et al. \cite{ABA}.  From this relation we can give the following combinatorial interpretation.
\begin{theorem}
Let $n$ and $k$ be non negative integers  and $s>0$.  Consider the set of $s$-tuples of permutations $(\sigma _1,\sigma _2,\dots , \sigma _s)$ such that $\min(\sigma _i)\subseteq \min (\sigma_{i-1})$ for all $i>1$ and $\sum _{j=1}^s|\min (\sigma _j)|=k+s-1.$ Then the number of such elements equals ${n\brack k}^{(s)}$.
\end{theorem}
\begin{proof}
By the recursion given in \eqref{recst1mod}, one has that
\begin{align*}
  {n\brack k}^{(s)}&=\sum _{\ell=0}^s{n-1\brack k-(s-\ell)}^{(s)}(n-1)^{\ell}=\sum _{\ell=0}^s{n-1\brack k-\ell}^{(s)}(n-1)^{s-\ell}.
\end{align*}
This recurrence corresponds to choosing if the last element of each permutation, i.e., $n$ is going to be fixed or not. Call $\ell$ the number of permutations where $n$ is going to be a fixed point. By the condition we imposed in the tuple, these have to be the first $\ell$ elements of the $s$-tuple. For the remaining $s-\ell$ elements of the tuple, we have to choose an element from the remaining $n-1$ to have $n$ as a preimage in each one of the permutations. We can do this in $(n-1)^{s-\ell}$ ways. This shows the claim because the initial condition is in $1-s,$ meaning we need $k-(1-s)=k+s-1$ cycles to fill.
\end{proof}

For example, take $n=3$ and $k=4,$ the following correspond to the $15$ tuples counted by ${3 \brack 4}^{(3)}$ having in total $3+4-1=6$ cycles.

$$( (1)(2)(3) , (1)(2,3) , (1,2,3) ),\quad
( (1)(2)(3) , (1)(2,3) , (1,3,2) ), $$
$$( (1)(2)(3) , (1,2)(3) , (1,2,3) ),\quad
( (1)(2)(3) , (1,2)(3) , (1,3,2) ), $$
$$( (1)(2)(3) , (1,3)(2) , (1,2,3) ),\quad
( (1)(2)(3) , (1,3)(2) , (1,3,2) ), $$
$$( (1)(2,3) , (1)(2,3) , (1)(2,3) ),\quad
( (1)(2,3) , (1)(2,3) , (1,3)(2) ), $$
$$( (1)(2,3) , (1,3)(2) , (1)(2,3) ),\quad
( (1)(2,3) , (1,3)(2) , (1,3)(2) ), $$
$$( (1,2)(3) , (1,2)(3) , (1,2)(3) ),\quad
( (1,3)(2) , (1)(2,3) , (1)(2,3) ), $$
$$( (1,3)(2) , (1)(2,3) , (1,3)(2) ),\quad
( (1,3)(2) , (1,3)(2) , (1)(2,3) ), $$
$$( (1,3)(2) , (1,3)(2) , (1,3)(2)).$$

\end{document}